\documentclass[proceedings]{aofa}

\usepackage[utf8]{inputenc}
\usepackage{subfigure}

\usepackage{amssymb,amsmath,mathrsfs}
\usepackage{color, colortbl}
\usepackage{tikz,pgflibraryshapes,pgf}
\usetikzlibrary{shapes,arrows,shadows,snakes}

\author[M. Fuchs and H.-K. Hwang]{M.
Fuchs\addressmark{1}\thanks{Partially supported by MOST under the
grant MOST-104-2923-M-009-006-MY3}\ and H.-K. Hwang\addressmark{2}}

\title[Dependencies in Random Tries]{Dependence between External
Path-Length and Size in Random Tries}

\address{\addressmark{1}Department of Applied Mathematics, National Chiao Tung University, Hsinchu, 300, Taiwan\\
\addressmark{2}Institute of Statistical Science, Academia Sinica, Taipei, 115, Taiwan}


\begin{document}

\newtheorem{thm}{Theorem}
\renewcommand*{\thethm}{\Alph{thm}}
\newtheorem{cor}{Corollary}
\newtheorem{lmm}{Lemma}
\newtheorem{conj}{Conjecture}
\newtheorem{pro}{Proposition}
\renewcommand*{\thepro}{\Alph{pro}}
\newtheorem{df}{Definition}
\newtheorem{Rem}{Remark}
\newcommand{\bigcell}[2]{\begin{tabular}{@{}#1@{}}#2\end{tabular}}

\maketitle

\begin{abstract}
\paragraph{Abstract.}

We study the size and the external path length of random tries and
show that they are asymptotically independent in the asymmetric case
but strongly dependent with small periodic fluctuations in the
symmetric case. Such an unexpected behavior is in sharp contrast to
the previously known results that the internal path length is totally
positively correlated to the size and that both tend to the same
normal limit law. These two examples provide concrete instances of
bivariate normal distributions (as limit laws) whose correlation is
$0$, $1$ and periodically oscillating.

\end{abstract}

\section{Introduction}\label{intro}

Tries are one of the most fundamental tree-type data structures in
computer algorithms. Their general efficiency depends on several
shape parameters, the principal ones including the depth, the height,
the size, the internal path-length (IPL), and the external
path-length (EPL); see below for a more precise description of those
studied in this paper. While most of these measures have been
extensively investigated in the literature, we are concerned here
with the question: \emph{how does the EPL depend on the size in a
random trie?} Surprisingly, while the IPL and the size are known to
have asymptotic correlation coefficient tending to one and to have
the same normal limit law after each being properly normalized (see
\cite{FuHwZa,FuLe}), this paper aims to show that the EPL exhibits a
completely different behavior depending on the parameter of the
underlying random bits being biased or unbiased. This is a companion
paper to \cite{ChFuHwNe}.

Given a sequence of binary strings (or keys), one can construct a
(binary) trie as follows. If $n=1$, then the trie consists of a
single root-node holding the sole string; otherwise, the root is used
to direct the strings into the corresponding subtree: if the first
bit of the input string is $0$ (or $1$), then the string goes to the
left (or right) subtree; strings going to the same subtree are then
constructed recursively in the same manner but instead of splitting
according to the first bit, the second bit of each string is then
used. In this way, a binary dictionary-type tree with two types of
nodes is constructed: external nodes for storing strings and internal
nodes for splitting the strings; see Figure~\ref{fg-trie} for a trie
of seven strings.

\begin{figure}[!ht]
\begin{center}
\begin{tikzpicture}[level/.style={sibling distance = 5cm/#1,
level distance = 0.75cm}, treenode/.style = {align=center,
inner sep=0pt, text centered}, arn_n/.style = {treenode, circle,
white, draw=black,fill=black, text width=0.6em},
arn_x/.style = {rounded corners=0.2ex,
drop shadow={shadow xshift=0.3ex,shadow yshift=-0.3ex},
fill=white,treenode, rectangle, draw=black,  minimum width=3em,
minimum height=0.7em, font=\scriptsize}]
\node [arn_n] {}
	child{ node [arn_n] {}  % lev 11
		child{ node [arn_x]
		{\colorbox{black}{\textcolor{white}{$00011100$}}}  % x6
			edge from parent
			node[left,yshift=0.15cm,font=\scriptsize] {$0$}
		}
		child{ node [arn_n] {} % lev 22
		child{ node [arn_x]
		{\colorbox{black}{\textcolor{white}{$01010100$}}}% x7
			edge from parent
			node[left,yshift=0.15cm,font=\scriptsize] {$0$}}	
			child{ node [arn_x]
			{\colorbox{black}{\textcolor{white}{$01100111$}}}% x5
			edge from parent
			node[right,yshift=0.15cm,font=\scriptsize] {$1$}}	
			edge from parent
			node[right,yshift=0.15cm,font=\scriptsize] {$1$}
            }
		edge from parent
		node[above,font=\scriptsize] {$0$}
	}
	child{ node [arn_n] {}	% lev 12
            child{ node [arn_x]
			{\colorbox{black}{\textcolor{white}{$10111010$}}} % x3
			edge from parent
			node[left,yshift=0.15cm,font=\scriptsize] {$0$}
            }
            child{ node [arn_n] {}
			child[]{ node [arn_n] {}
				child[]{ node [arn_n] {}
					child[sibling distance = 1.5cm]{ node [arn_x]
					{\colorbox{black}{\textcolor{white}{$11000011$}}}
					edge from parent node[left,yshift=0.10cm,
					font=\scriptsize] {$0$}} %x4
					child[sibling distance = 1.5cm]{ node [arn_n] {}
						child[sibling distance = 1.8cm]{ node [arn_x]
						{\colorbox{black}{\textcolor{white}
						{$11001000$}}}
						edge from parent node[left,yshift=0.00cm,
						font=\scriptsize] {$0$}}	% x2
						child[sibling distance = 1.5cm]{ node [arn_x]
						{\colorbox{black}{\textcolor{white}
						{$11001010$}}} edge from parent node[right,
						yshift=0.00cm,font=\scriptsize] {$1$}}	% x1
					edge from parent node[right,yshift=0.10cm,
					font=\scriptsize] {$1$}
					}
				edge from parent node[left,yshift=0.10cm,
				font=\scriptsize] {$0$}
				}
				child[missing]{ node [arn_n] {}}
				edge from parent node[left,yshift=0.10cm,
				font=\scriptsize] {$0$}
			}
			child[missing]{ node [arn_r] {}}
			edge from parent
			node[right,yshift=0.15cm,font=\scriptsize] {$1$}
            }
		edge from parent
		node[above,font=\scriptsize] {$1$}
	};
\end{tikzpicture}
\caption{\emph{A trie with $n=7$ records: the (filled) circles
represent internal nodes and rectangles holding the binary strings
are external nodes. In this example, $S_n=8$, $K_n=27$, and
$N_n=18$.}}
\label{fg-trie}
\end{center}
\end{figure}

The random trie model we consider here assumes that each of the $n$
binary keys is an infinite sequence consisting of independent
Bernoulli bits each with success probability $0<p<1$. Then the trie
constructed from this sequence is a random trie. We define three
shape parameters in a random trie of $n$ strings:
\begin{itemize}
	\item size $S_n$: the total number of internal nodes used;
	\item IPL (or node path-length, NPL) $N_n$: the sum
	      of the distances between the root to each internal node;
    \item EPL (or key path-length, KPL) $K_n$: the sum
	      of the distances between the root to each external node.
\end{itemize}
We will use mostly NPL in place of IPL, and KPL in place of EPL, the
reason being an easier comparison with the corresponding results in
random $m$-ary search trees in the companion paper \cite{ChFuHwNe};
see below for more details.

By the recursive definition, we have the following recurrence
relations
\begin{align}\label{dist-rec}
	\begin{cases}
		S_n\stackrel{d}{=}S_{B_n}+S_{n-B_n}^{*}+1,\\
		K_n\stackrel{d}{=}K_{B_n}+K_{n-B_n}^{*}+n,\\
		N_n\stackrel{d}{=}N_{B_n}+N_{n-B_n}^{*}
		+S_{B_n}+S_{n-B_n}^{*},
	\end{cases}\qquad(n\ge2),
\end{align}
where $B_n=\text{Binom}(n,p)$ and $S_0=S_1=K_0=K_1 =N_0=N_1=0$. Here
$(S_n^*), (K_n^*)$, and $(N_n^*)$ are independent copies of $(S_n),
(K_n)$ and $(N_n)$, respectively. While many stochastic properties of
these random variables are known (see \cite{FuHwZa} and the references
therein), much less attention has been paid to their correlation and
dependence structure.

The asymptotic behaviors of the moments of random variables defined on
tries typically depend on the ratio $\frac{\log p}{\log q}$ being
rational or irrational, where $q=1-p$. So we introduce, similar to
\cite{FuHwZa}, the notation
\begin{align}\label{gk}
	\mathscr{F}[g](z)
	=\begin{cases}{
	    \sum_{k\in\mathbb{Z}}g_k
		z^{-\chi_k}},&{ \text{if\ }
		\frac{\log p}{\log q}\in\mathbb{Q}};\\
		g_0,&{
		\text{if\ } \frac{\log p}{\log q}\not\in\mathbb{Q}},
	\end{cases}
\end{align}
where $g_k$ represents a sequence of coefficients and $\chi_k=
\frac{2rk\pi i}{\log p}$ when $\frac{\log p}{\log q}= \frac rl$ with
$r$ and $l$ coprime. In simpler words, $\mathscr{F}[g](z)$ is a
periodic function in the rational case, and a constant in the
irrational case. We also use $\mathscr{F}[\cdot](z)$ as a generic
symbol if the exact form of underlying sequence matters less, and in
this case each occurrence may not represent the same function.

With this notation, the asymptotics of the mean and the variance of the above three shape parameters are
summarized in the following table; see \cite{FuHwZa} and the references
therein for more information.
\begin{table}[!h]
\begin{center}
\begin{tabular}{|c||c|c|} \hline
	Shape parameters & $\frac1n(\text{mean})\sim$
	& $\frac1n(\text{variance})\sim$  \\ \hline
	Size $S_n$ & $\mathscr{F}[\cdot](n)$
	& $\mathscr{F}[g^{(1)}](n)$ \\ \hline
	NPL $N_n$ & $\frac{\mathbb{E}(S_n)}{n}\cdot \frac{\log n}{h}
	$ & $\frac{\mathbb{V}(S_n)}{n}
	\cdot \frac{(\log n)^2}{h^2}$ \\ \hline
%	\rowcolor{gray!10}
	KPL $K_n$ & $\frac{\log n}{h}+\mathscr{F}[\cdot](n)$
	& \textcolor{red}{
	$\frac{pq\log^2\frac pq}{h^2}$}
	$\cdot\frac{\log n}{h}+\mathscr{F}[g^{(3)}](n)$
	\\ \hline \hline
%	\rowcolor{gray!10}
	Depth $D_n$ & $\mathbb{E}(D_n) =
	\frac{\mathbb{E}(K_n)}{n}$ &
	$\mathbb{V}(D_n) =
	\frac{\mathbb{V}(K_n)}{n}+O(1)$
	\\ \hline
\end{tabular}
\end{center}
\vspace*{-.4cm}
\caption{\emph{Asymptotic patterns of the means and the variances of
the shape parameters discussed in this paper. Here
$\mathscr{F}[\cdot](n)$ differs from one occurrence to another and
$h=-p\log p-q\log q$ denotes the entropy. Expressions for $g^{(1)}_k$
and $g^{(3)}_k$ will be given below. Asymptotic normality holds for
all three random variables $S_n, N_n, K_n$.}}
\label{tb-SKN}
\end{table}

Note specially that the leading constant
\[
    \lambda=\lambda_p := \frac{pq\log^2\frac pq}{h^3}
	= \frac{(p\log^2p+q\log^2q)-h^2}{h^3}
\]
in the asymptotic approximation to $\mathbb{V}(K_n)$ equals zero when
$p=q$, implying that $\mathbb{V}(K_n)$ is not of order $n\log n$ but
of linear order in the symmetric case. \emph{This change of order can
be regarded as the source property distinguishing the dependence and
independence of $K_n$ on $S_n$.}

On the other hand, if we denote by $D_n$ the depth, which is defined
to be the distance between the root and a randomly chosen external
node (each with the same probability), then we have not only the
relation $\mathbb{E}(D_n)n = \mathbb{E}(K_n)$, but also the
asymptotic equivalent $\mathbb{V}(D_n)n \sim \mathbb{V}(K_n)$ when
$p\neq 1/2$ (or $\lambda>0$), and a central limit theorem holds; see
Devroye \cite{De98}.

From Table~\ref{tb-SKN}, we see roughly that each internal node
contributes $\frac{\log n}h$ to $N_n$, namely, that $N_n\approx
S_n\cdot \frac{\log n}h$. Indeed, it was proved in \cite{FuHwZa} that
the correlation coefficient of $S_n$ and $N_n$ satisfies
\begin{align}\label{rho-Sn-Nn}
    \rho(S_n,N_n)\sim 1\qquad(0<p<1).
\end{align}
Such a linear correlation was further strengthened in \cite{FuLe},
where it was proved that both random variables tend to the \emph{same}
normal limit law $\mathcal{N}_1$ (with zero mean and unit variance)
\[
	\biggl(\frac{S_n-\mathbb{E}(S_n)}
	{\sqrt{\mathbb{V}(S_n)}},\frac{N_n-\mathbb{E}(N_n)}
	{\sqrt{\mathbb{V}(N_n)}}\biggr)
	\stackrel{d}{\longrightarrow}
	(\mathcal{N}_1,\mathcal{N}_1),
\]
where $\stackrel{d}{\longrightarrow}$ denotes convergence in
distribution. In terms of the bivariate normal law $\mathcal{N}_2$
(see Tong \cite{To90}), we can write
\[
	\biggl(\frac{S_n-\mathbb{E}(S_n)}
	{\sqrt{\mathbb{V}(S_n)}},\frac{N_n-\mathbb{E}(N_n)}
	{\sqrt{\mathbb{V}(N_n)}}\biggr)^{\intercal}
	\stackrel{d}{\longrightarrow} \mathcal{N}_2(0,E_2),
\]
where $E_2=\scriptsize \begin{pmatrix} 1 & 1\\ 1& 1\end{pmatrix}$ is
a singular matrix and $\mathbf{A}^\intercal$ denotes the transpose of
matrix $\mathbf{A}$.

We show that the correlation and dependence of $K_n$ on $S_n$ are
drastically different. We start with their correlation coefficient.

\begin{thm}\label{main-thm-1}
The covariance of the number of internal nodes and KPL in a random
trie of $n$ strings satisfies
\[
    \mathrm{Cov}(S_n,K_n) \sim n \mathscr{F}[g^{(2)}](n),
\]
where $g^{(2)}_k$ is given in Proposition \ref{fc-cov} below, and
their correlation coefficient satisfies
\begin{align}\label{rho-Sn-Kn}
	\rho(S_n,K_n)\sim
	\begin{cases}
		0,&\text{if}\ p\ne \frac12\\
		F(n),&\text{if}\ p=\frac12.
	\end{cases}
\end{align}
Here $F(n)=\frac{\mathscr{F}[g^{(2)}](n)}
{\sqrt{\mathscr{F}[g^{(1)}](n) \mathscr{F}[g^{(3)}](n)}}$ is a
periodic function with average value $0.927\cdots$.
\end{thm}
The result \eqref{rho-Sn-Kn} is to be compared with \eqref{rho-Sn-Nn}
(which holds for all $p\in(0,1)$): \emph{the surprising difference
here comes not only from the (common) distinction between $p=\frac12$
and $p\ne \frac12$ but also from the (less expected) intrinsic
asymptotic nature.}

\begin{figure}[!h]
\begin{center}
\hspace*{0 cm}
\begin{tikzpicture}[]
\node[align=center] at (0,0) {
\includegraphics[height=3.8cm]{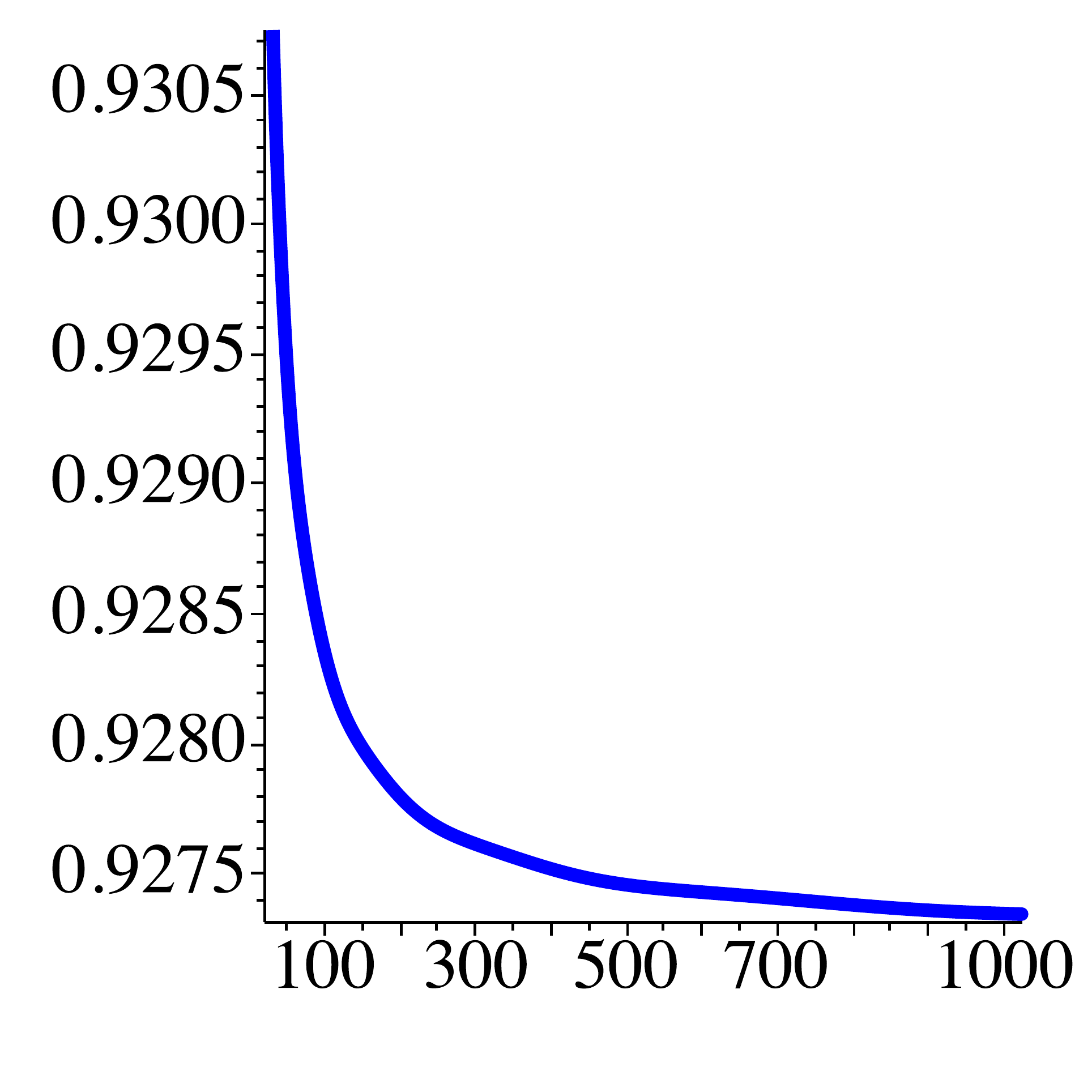}\hspace*{0 pt}
\includegraphics[height=3.8cm]{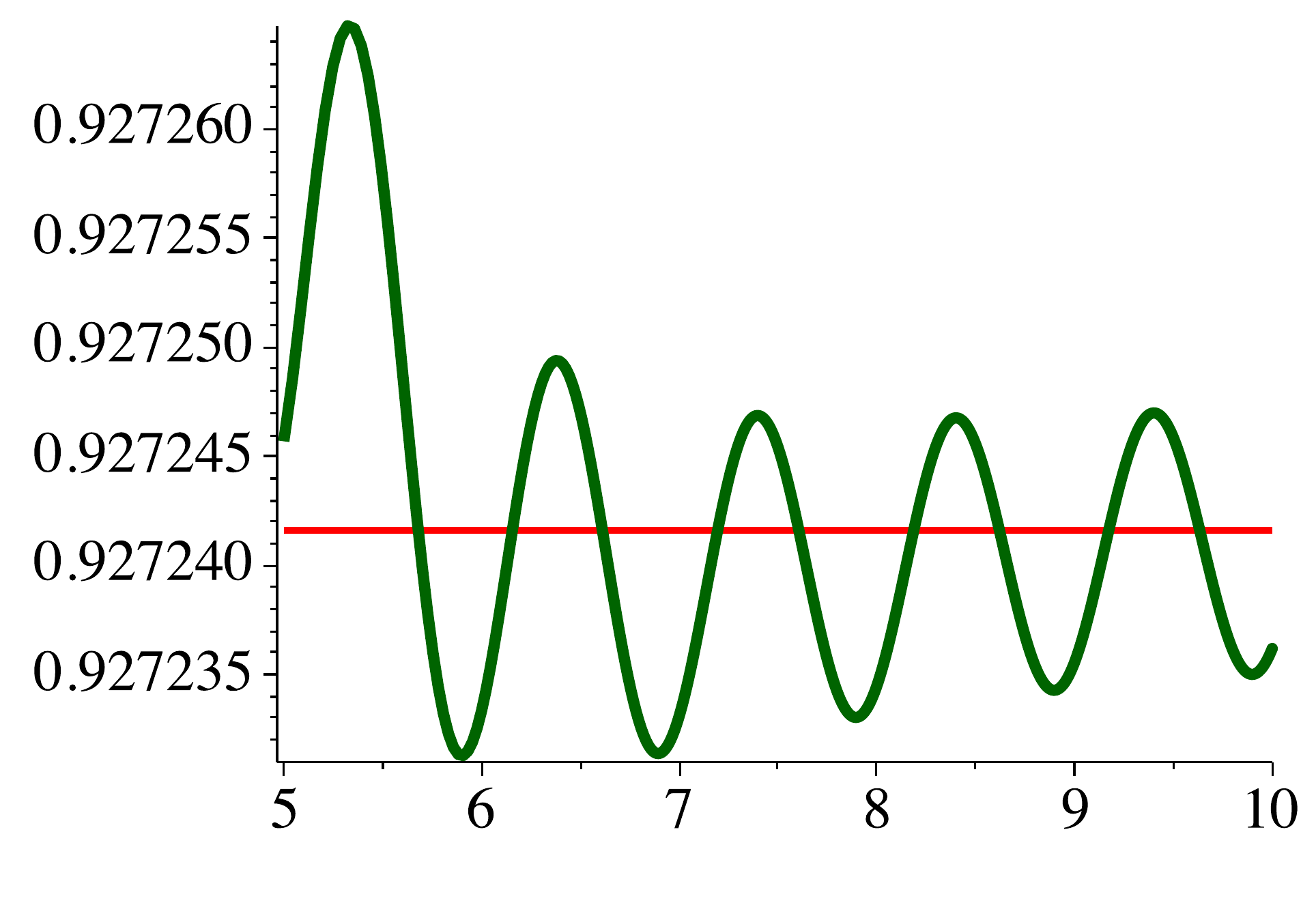}\hspace*{0 pt}
\includegraphics[height=3.8cm]{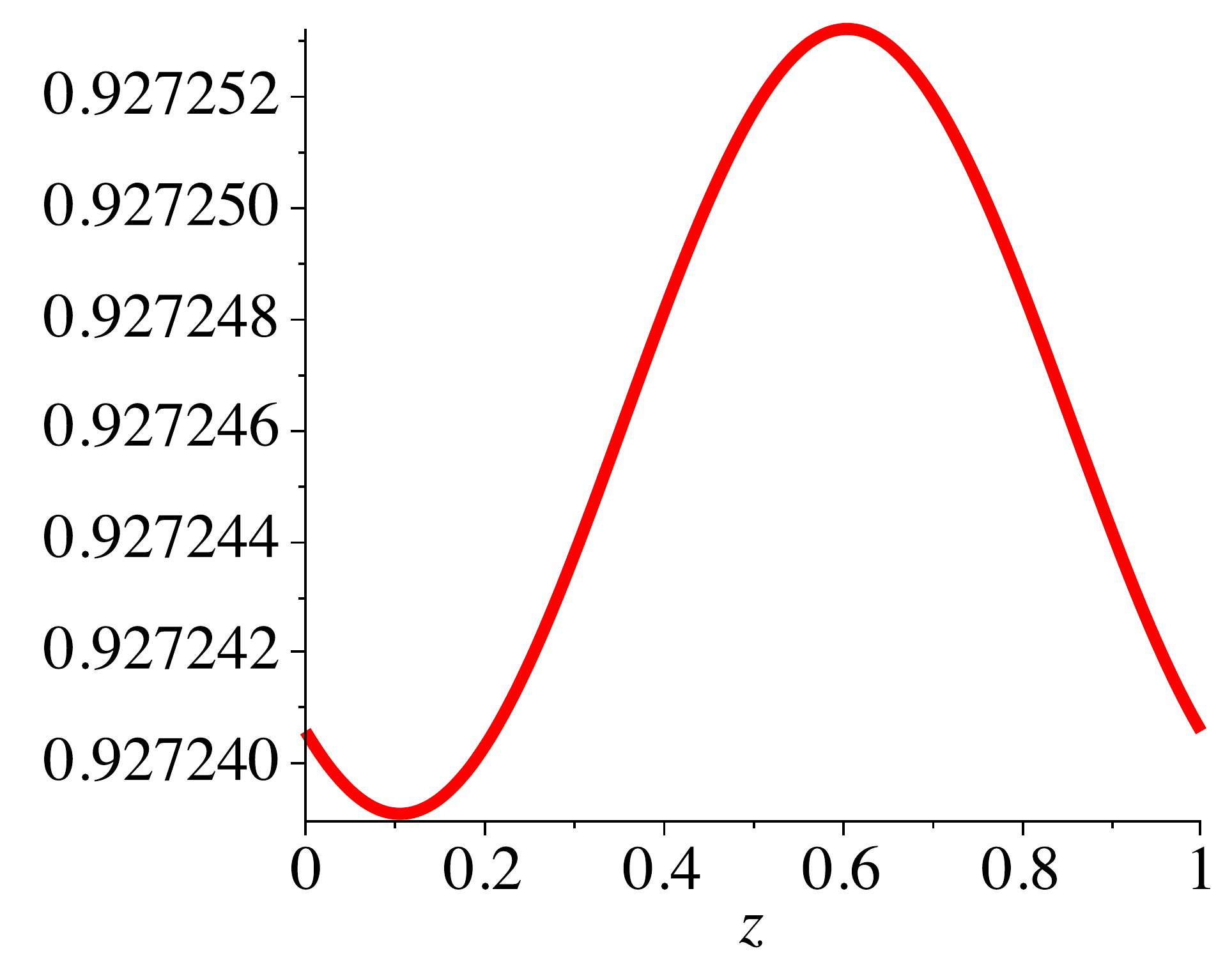}\hspace*{0 pt}
};
\end{tikzpicture}
\end{center}
\vspace*{-.4cm}
\caption{\emph{$p=\frac12$: periodic fluctuations of (i)
$\rho(S_n,K_n)$ (left) for $n=32,\dots, 1024$, (ii)
$\frac{\mathrm{Cov}(S_n,K_n)}{\sqrt{\mathbb{V}(S_n)
(\mathbb{V}(K_n)+1.046)}}$ (middle) in logarithmic scale, and
(iii) $F(n)$ by its Fourier series expansion (right). Note that
the fluctuations are only visible by proper corrections either
in the denominator or in the numerator because the amplitude of
$F$ is very small: $|F(\cdot)|\le 1.5\times 10^{-5}$.}}
\end{figure}

Furthermore, we show that this different behavior cannot be ascribed
to the weak measurability of nonlinear dependence of Pearson's
correlation coefficient $\rho$ since the same dependence is also
present in the limiting distribution. (For the univariate central
limit theorems implied by the result below, see Jacquet and
R\'{e}gnier \cite{JaRe} where such results were first established.)

\begin{thm}\label{main-thm-2}
\begin{itemize}
\item[(i)] For $p\ne \frac12$, we have
\[
	\biggl(\frac{S_n-\mathbb{E}(S_n)}
	{\sqrt{\mathbb{V}(S_n)}},\frac{K_n-\mathbb{E}(K_n)}
	{\sqrt{\mathbb{V}(K_n)}}\biggr)^{\intercal}
	\stackrel{d}{\longrightarrow} \mathcal{N}_2(0,I_2),
\]
where $I_2$ denotes the $2\times 2$ identity matrix.
\item[(ii)] For $p=\frac12$, we have
\[
	\Sigma_n^{-\frac12}
	\begin{pmatrix}
		S_n-\mathbb{E}(S_n) \\
		K_n-\mathbb{E}(K_n)
	\end{pmatrix}
		\stackrel{d}{\longrightarrow}\mathcal{N}_2(0,I_2),
\]
where $\Sigma_n$ denotes the (asymptotic) covariance matrix of $S_n$
and $K_n$:
\[
    \Sigma_n := n\begin{pmatrix}
	    \mathscr{F}[g^{(1)}](n)
		& \mathscr{F}[g^{(2)}](n) \\
		\mathscr{F}[g^{(2)}](n)
		& \mathscr{F}[g^{(3)}](n)
	\end{pmatrix}.
\]
\end{itemize}
\end{thm}
Alternatively, we may define $\Sigma_n:=n\begin{pmatrix}
\mathscr{F}[g^{(1)}](n) & \mathscr{F}[g^{(2)}](n) \\
\mathscr{F}[g^{(2)}](n) & \lambda\log n+\mathscr{F}[g^{(3)}](n)
\end{pmatrix}$. Then both cases can be stated in one as
$
	\Sigma_n^{-\frac12}
	\begin{pmatrix}
		S_n-\mathbb{E}(S_n) \\
		K_n-\mathbb{E}(K_n)
	\end{pmatrix}
		\stackrel{d}{\longrightarrow}\mathcal{N}_2(0,I_2).
$
On the other hand, since for bivariate normal distribution, zero
correlation implies independence (see \cite{To90}), it is more
transparent to split the statement into two cases. See
Figure~\ref{fig-SK} for 3D-plots of the joint distributions of
$(S_n,K_n)$ when $n=10^7$.

\begin{figure}[!h]
\begin{center}
\hspace*{0 cm}
\begin{tikzpicture}[]
\node[align=center] at (0,0) {
\includegraphics[width=4cm]{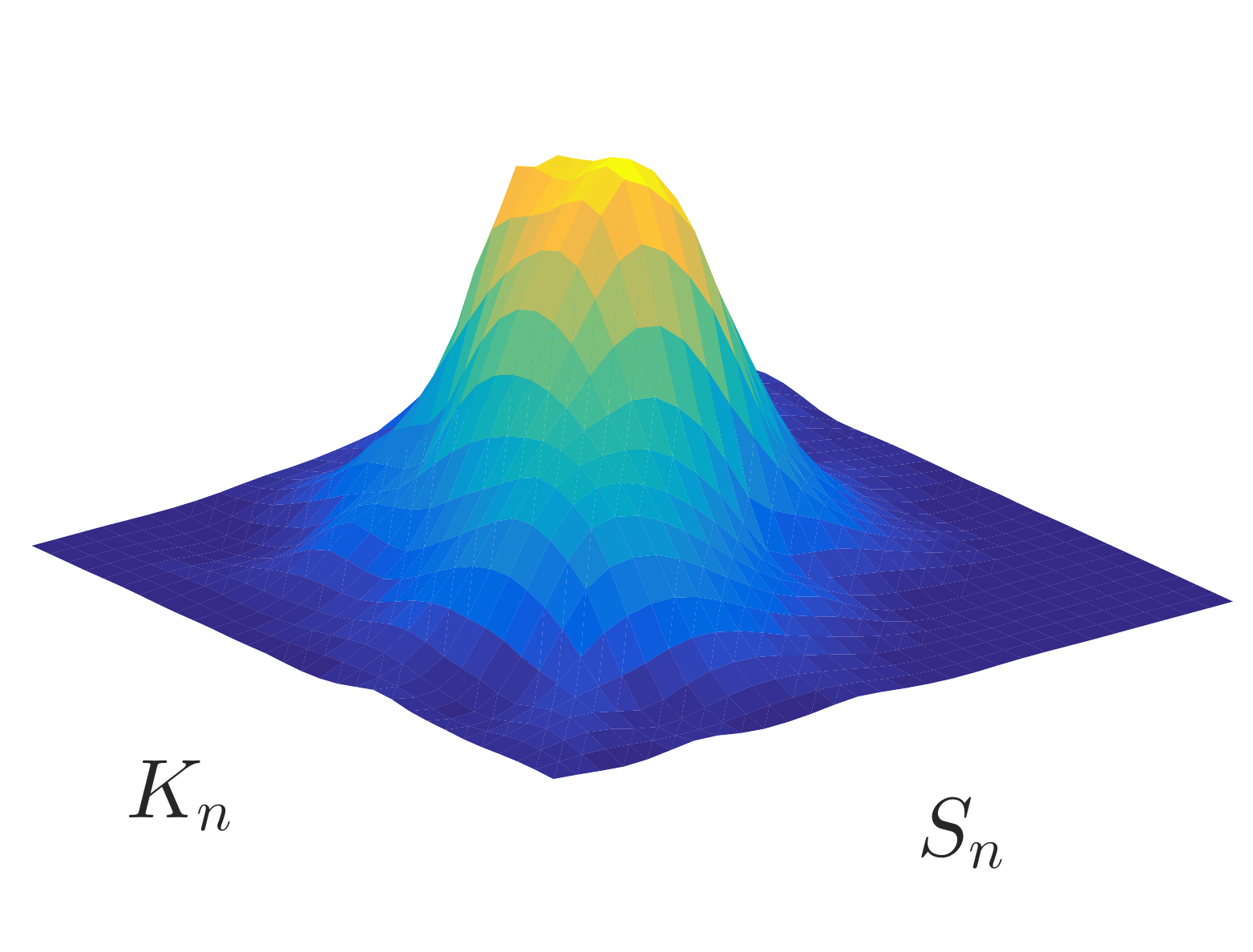}\hspace*{0 pt}
\includegraphics[width=4cm]{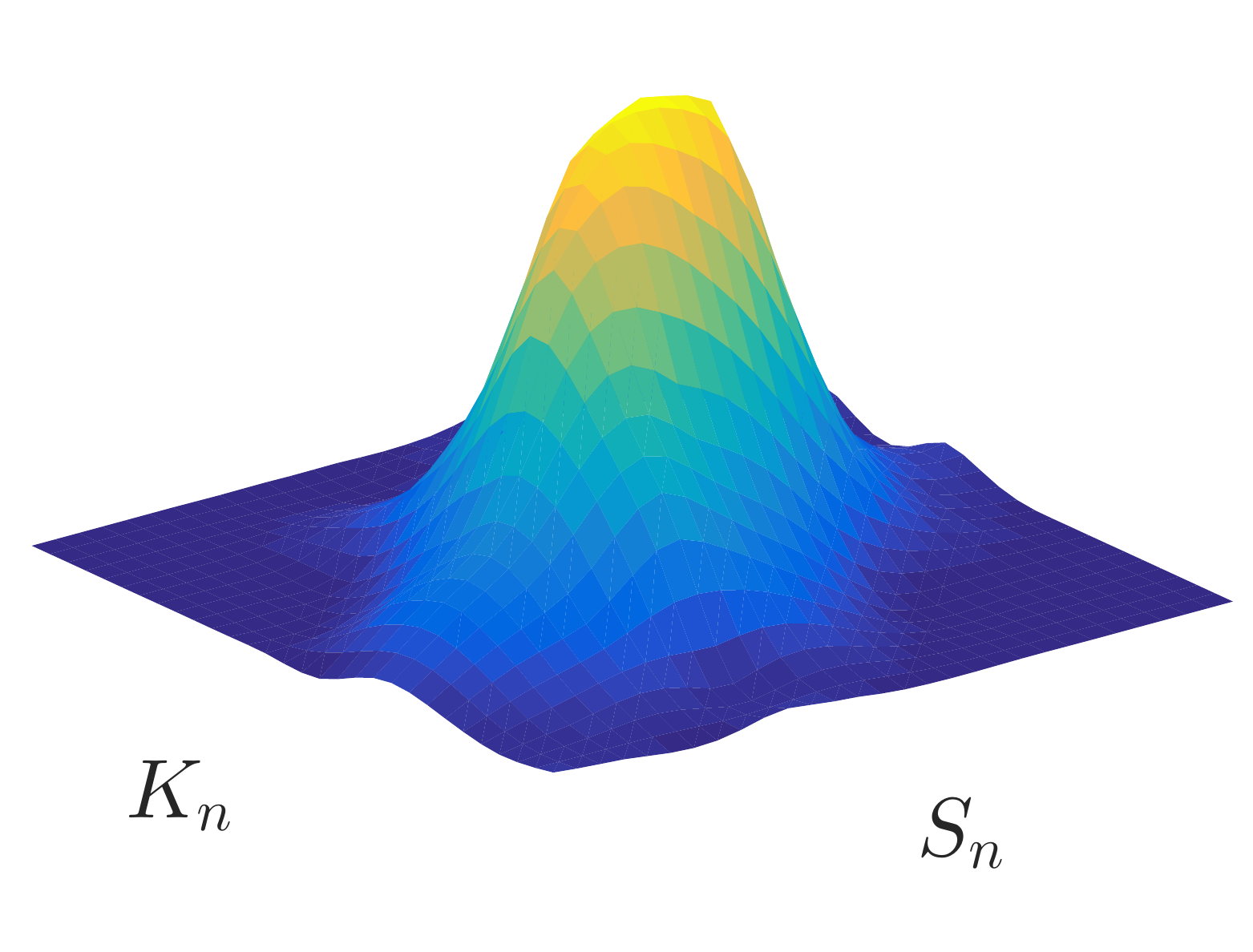}\hspace*{0 pt}
\includegraphics[width=4cm]{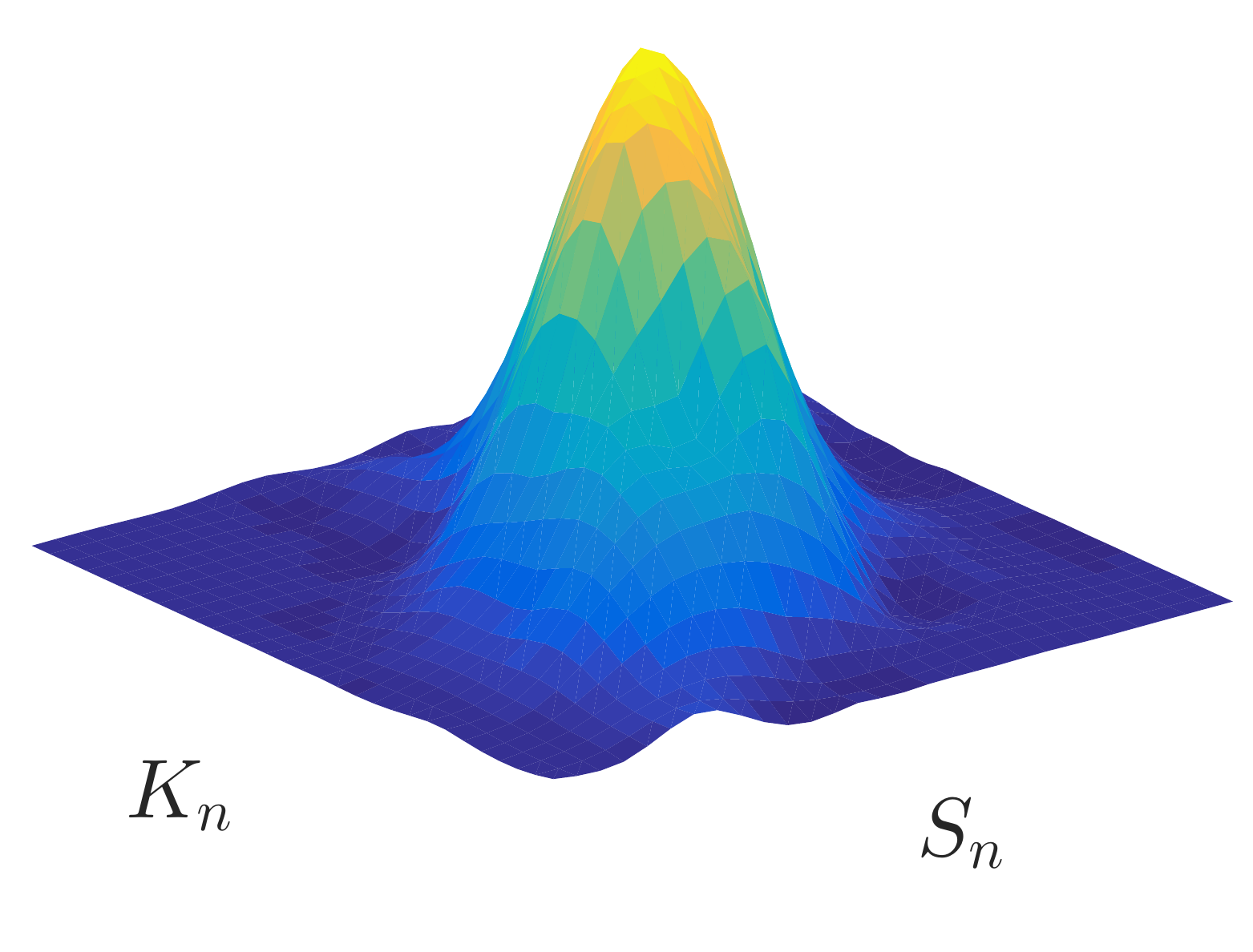}\hspace*{0 pt}\\\\
\includegraphics[width=4cm]{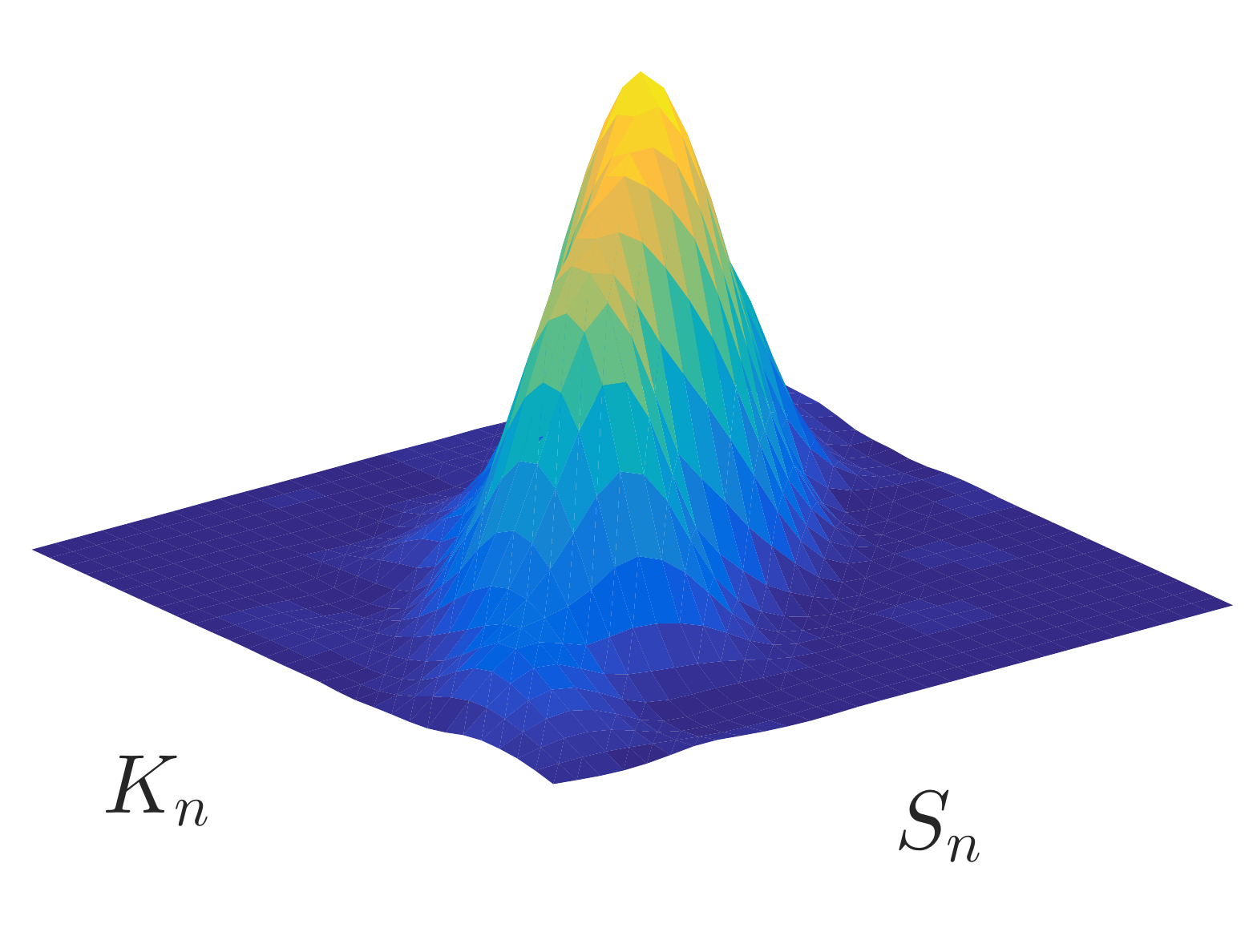}
\includegraphics[width=4cm]{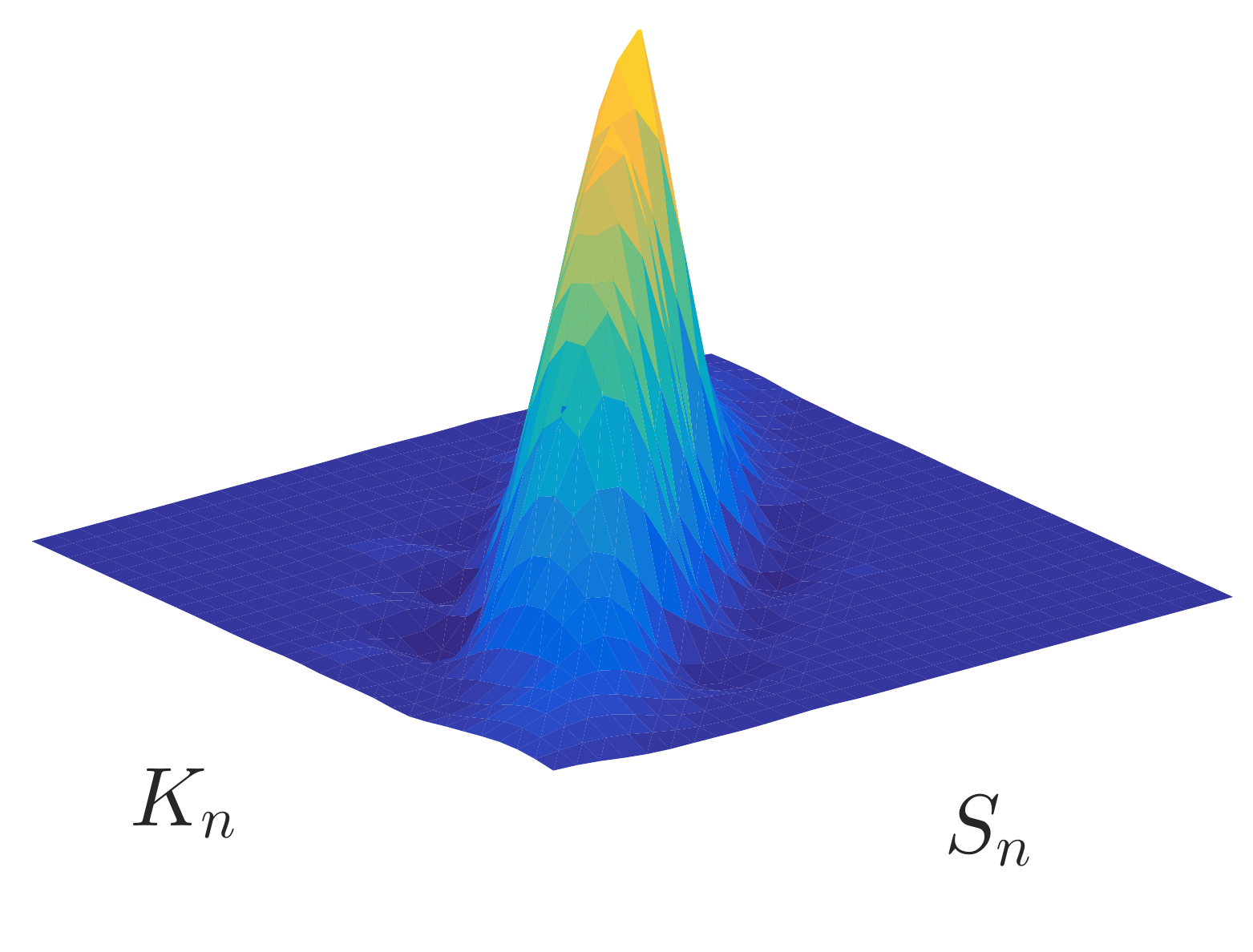}\hspace*{0 pt}
\includegraphics[width=4cm]{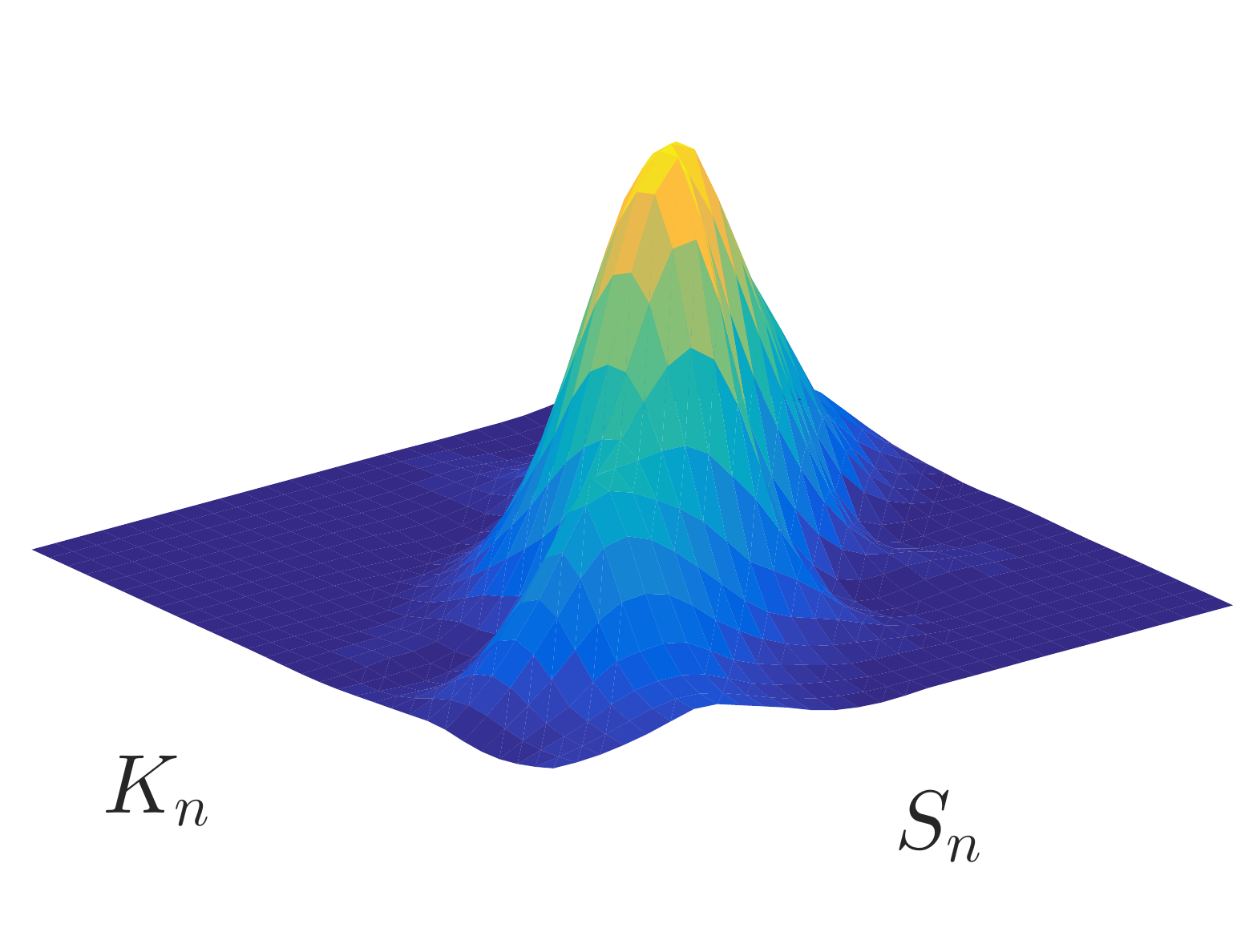}\hspace*{0 pt}\\\\
\includegraphics[width=4cm]{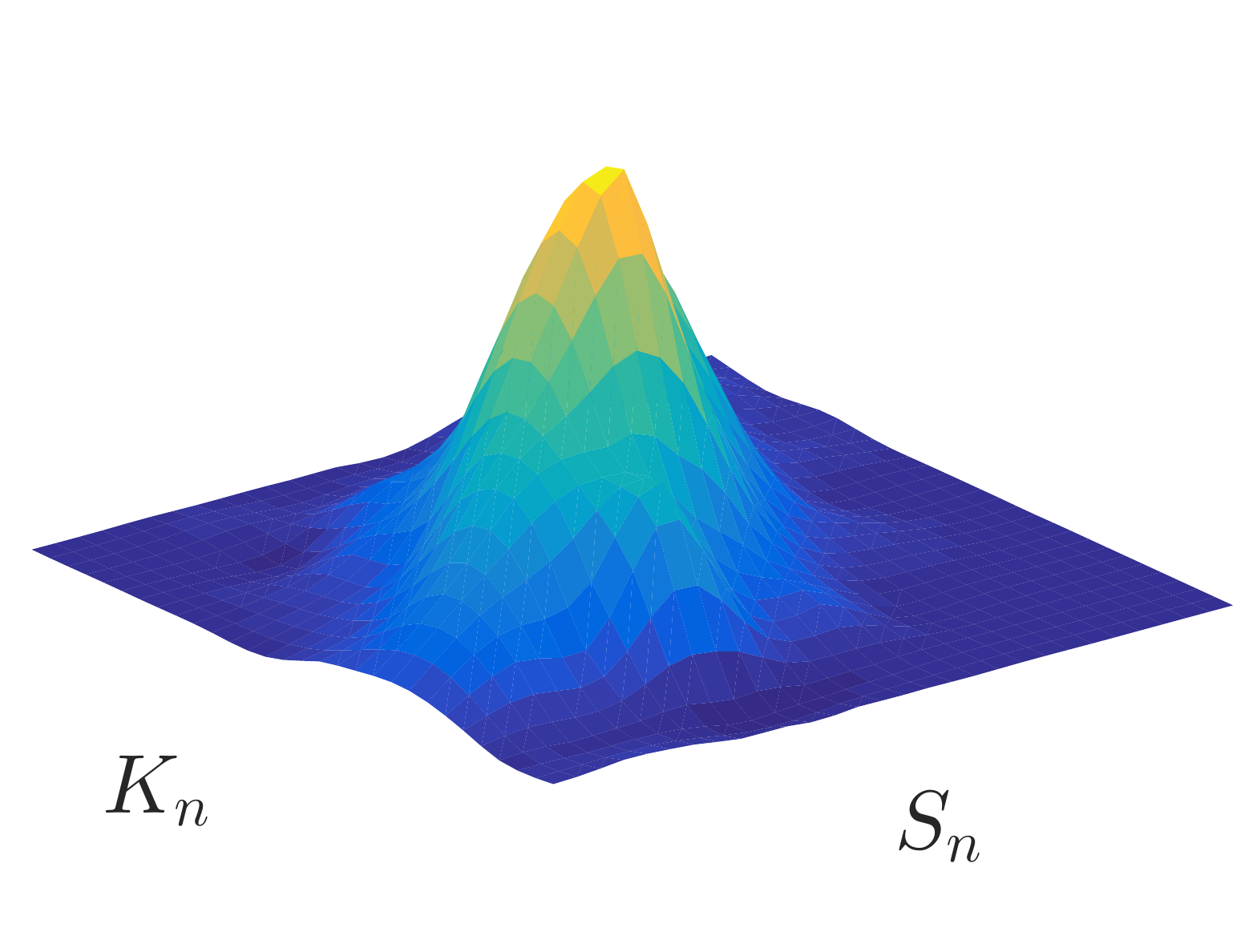}\hspace*{0 pt}
\includegraphics[width=4cm]{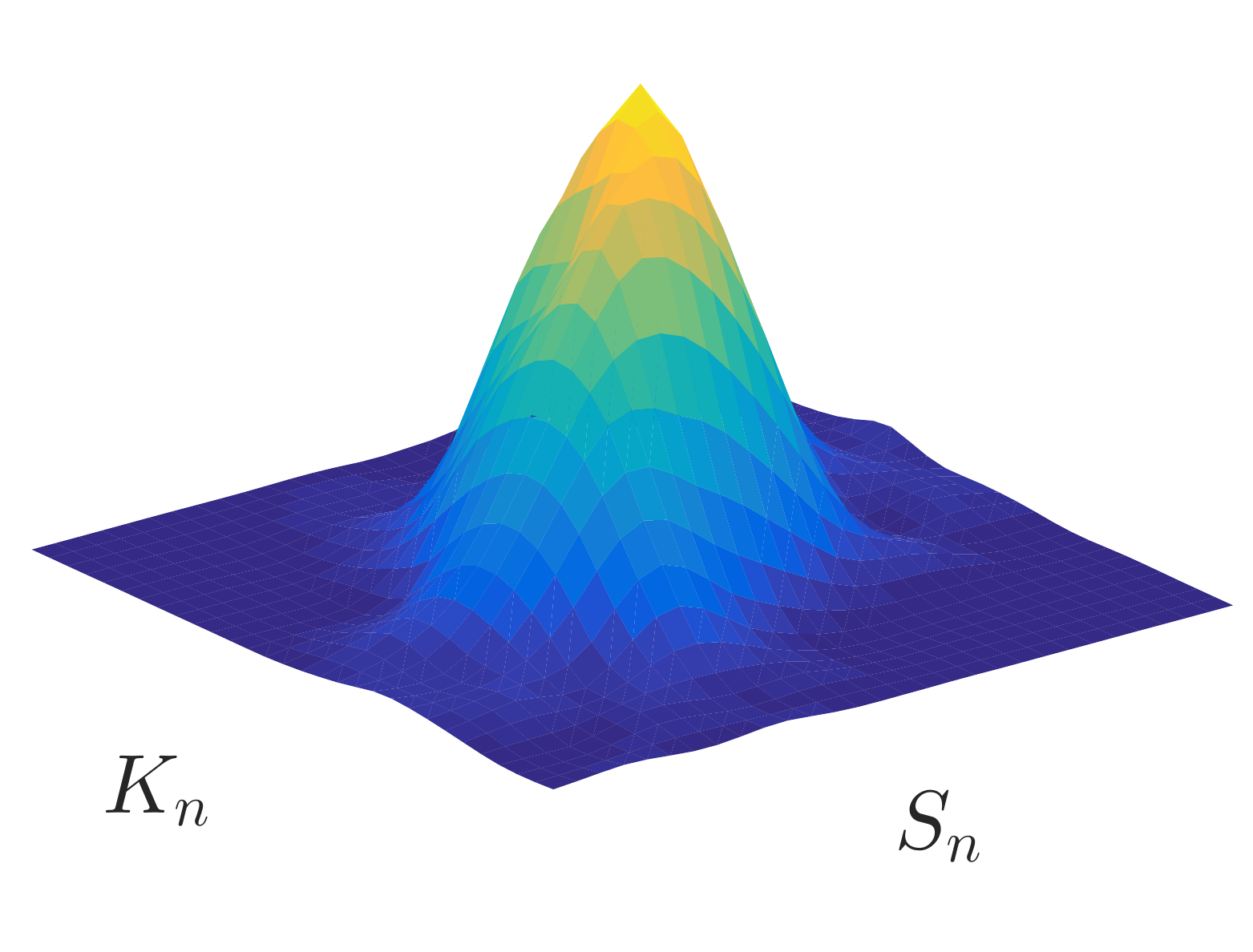}\hspace*{0 pt}
\includegraphics[width=4cm]{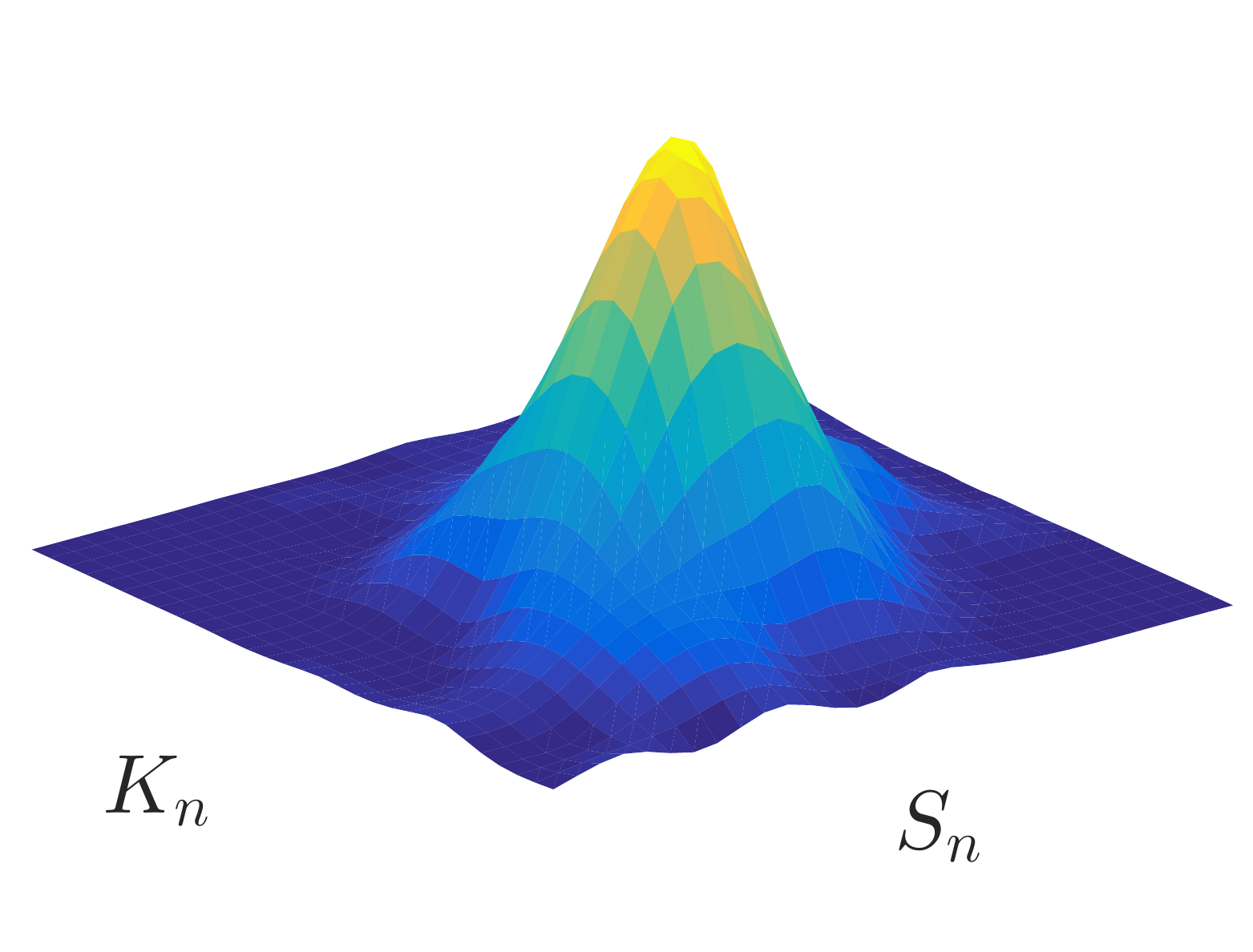}\hspace*{0 pt}
};
\node[] (p04) at (-5.15,1.1) {\tiny $p=0.4$};
\node[] (p05) at (-1,1.1) {\tiny $p=0.5$};
\node[] (p06) at (3.15,1.1) {\tiny $p=0.6$};

\node[] (p01) at (-5.15,4.75) {\tiny $p=0.1$};
\node[] (p02) at (-1,4.75) {\tiny $p=0.2$};
\node[] (p03) at (3.15,4.75) {\tiny $p=0.3$};

\node[] (p07) at (-5.15,-2.5) {\tiny $p=0.7$};
\node[] (p08) at (-1,-2.5) {\tiny $p=0.8$};
\node[] (p09) at (3.15,-2.5) {\tiny $p=0.9$};
\end{tikzpicture}
\end{center}
\vspace*{-.5cm}
\caption{\emph{Joint distributions of $(S_n,K_n)$ by Monte-Carlo
simulations for $n=10^7$ and varying $p$: the case $p=0.5$
is seen to have stronger dependence than the others.}}
\label{fig-SK}
\end{figure}

These results are to be compared with the corresponding ones for
random $m$-ary search trees \cite{ChFuHwNe}, and the differences
for correlation coefficients are summarized in
Table~\ref{tab-tries-mst}.
\begin{table}[!h]
\begin{center}
\begin{tabular}{|c||c|c|}
\hline
trees & $\rho(S_n, K_n)$ & $\rho(S_n, N_n)$ \\ \hline
tries & \multicolumn{1}{c|}{$\left\{\begin{tabular}{@{\ }l@{}}
   $p\neq q:\rightarrow 0$ \\ $p=q:\text{periodic}$
   \end{tabular}\right.$} & $\sim1$ \\ \hline
   \bigcell{c}{m-ary \\ search trees} &
   \multicolumn{2}{c|}{$\left\{\begin{tabular}{@{\ }c@{}}
   $3\le m\le26:\rightarrow0$  \\ $m\ge 27:\text{periodic}$
   \end{tabular}\right.$}   \\ \hline
\end{tabular}
\end{center}
\vspace*{-.5cm}
\caption{\emph{A comparison of the correlation coefficients
for random tries and random $m$-ary search trees:
the size of $m$-ary search trees corresponds to the space
requirement, and the KPL and NPL are defined similarly as
in tries.}} \label{tab-tries-mst}
\end{table}
Furthermore, the joint distribution for $m$-ary search trees
undergoes a phase change at $m=26$: if the branching factor $m$
satisfies $3\le m\le 26$, then the space requirement is
asymptotically independent from KPL and NPL, while for $m\ge27$,
their limiting joint distributions contain periodic fluctuations and
are dependent; see \cite{ChFuHwNe} for more information.

Finally, similar results as those in this paper also hold for other
digital families of trees, but for simplicity we focus on tries in
this paper; see \cite{HwFuZa,FuHwZa} for more references.

\section{Covariance and Correlation Coefficient}

In this section, we sketch the main ideas leading to the proof of
Theorem \ref{main-thm-1} on the asymptotics of the covariance and
correlation coefficient of $S_n$ and $K_n$. For the latter, we also
need the variances of $S_n$ and $K_n$ which have been known for a
long time; see Jacquet and R\'{e}gnier \cite{JaRe}, Kirschenhofer and
Prodinger \cite{KiPr}, Kirschenhofer et al. \cite{KiPrSz},
R\'{e}gnier and Jacquet \cite{ReJa} or the recent paper
\cite{FuHwZa}. (See also Table \ref{tb-SKN} for a summary of these
results.)

Our method of proof is based on the by-now standard two-stage
approach relying on the theory of analytic de-Poissonization and
Mellin transform whose origin can be traced back to Jacquet and
R\'{e}gnier \cite{JaRe}. See Flajolet et al. \cite{FlGoDu} for a
survey on Mellin transform, and Jacquet and Szpankowski \cite{JaSz}
for a survey on analytic de-Poissonization. For the computation of
the covariance, the manipulation can be largely simplified by the
additional notions of Poissonized variance and admissible functions
further developed in our previous papers \cite{FuHwZa,HwFuZa}.

The starting point of our analysis is the recurrence satisfied by
$S_n$ and $K_n$ in (\ref{dist-rec}). A standard means in the
computation of moments of $S_n$ and $K_n$ is the Poisson generating
function, which corresponds to the moments of $S_n$ and $K_n$ with
$n$ replaced by a Poisson random variable with parameter $z$ (this
step is called \emph{Poissonization}).

More precisely, define the Poisson generating function of
$\mathbb{E}(S_n)$ and that of $\mathbb{E}(K_n)$: $\tilde{f}_{1,0}(z)
:=e^{-z}\sum_{n\ge 0} \mathbb{E}(S_n)\frac{z^n}{n!}$ and
$\tilde{f}_{0,1}(z) :=e^{-z}\sum_{n\ge 0}\mathbb{E}(K_n)
\frac{z^n}{n!}$. Then the recurrences \eqref{dist-rec} lead to the
functional equations
\begin{align}\label{f10-f01}
	\begin{cases} \tilde{f}_{1,0}(z)
	=\tilde{f}_{1,0}(pz)+\tilde{f}_{1,0}(qz)+1-(1+z)e^{-z},\\
	\tilde{f}_{0,1}(z)
    =\tilde{f}_{0,1}(pz)+\tilde{f}_{0,1}(qz)+z(1-e^{-z}).
	\end{cases}
\end{align}
From these equations, we obtain, by Mellin transform techniques
\cite{FlGoDu},
\begin{equation}\label{exp-means}
	\tilde{f}_{1,0}(z)
	\sim z\mathscr{F}[\cdot](z),
	\qquad\text{and}\qquad
	\tilde{f}_{0,1}(z)
	\sim h^{-1}z\log z+z\mathscr{F}[\cdot](z),
\end{equation}
for large $|z|$ in the half-plane $\Re(z)\ge\varepsilon>0$, where $h$
denotes the entropy of Bernoulli($p$). Then, by Cauchy's integral
representation and analytic de-Poissonization techniques \cite{JaSz},
we obtain precise asymptotic approximations to $\mathbb{E}(S_n)$ and
to $\mathbb{E}(K_n)$.

Similarly, for the variances $\mathbb{V}(S_n)$ and $\mathbb{V}(K_n)$,
we introduce the Poisson generating functions of the second moments:
$\tilde{f}_{2,0}(z) :=e^{-z}\sum_{n\ge0} \mathbb{E}(S_n^2)
\frac{z^n}{n!}$ and $\tilde{f}_{0,2}(z) :=e^{-z}\sum_{n\ge
0}\mathbb{E}(K_n^2) \frac{z^n}{n!}$, which then satisfy, by
\eqref{dist-rec}, the same type of functional equations as in
\eqref{f10-f01} but with different non-homogeneous parts. Instead of
computing directly asymptotic approximations to the second moments, it
proves computational more advantageous to consider the Poissonized
variances
\begin{align}\label{VS-VK}
    \begin{cases}
		\tilde{V}_S(z)
		:= \tilde{f}_{2,0}(z)-\tilde{f}_{1,0}(z)^2
		-z\tilde{f}_{1,0}'(z)^2,\\
		\tilde{V}_K(z)
		:= \tilde{f}_{0,2}(z)-\tilde{f}_{0,1}(z)^2
		-z\tilde{f}_{0,1}'(z)^2,
	\end{cases}	
\end{align}
and then following the same Mellin-de-Poissonization approach (as for
the means) to derive the first and the third asymptotic estimate in
the second column of Table~\ref{tb-SKN}. It remains to derive the
claimed estimate for the covariance. For that purpose, we then
introduce the Poisson generating function $\tilde{f}_{1,1}(z)
:=e^{-z}\sum_{n\ge 0}\mathbb{E}(S_nK_n)\frac{z^n}{n!}$, which
satisfies, again by \eqref{dist-rec},
\begin{align*}
	\tilde{f}_{1,1}(z)
	&=\tilde{f}_{1,1}(pz)+\tilde{f}_{1,1}(qz)
	+\tilde{f}_{1,0}(pz)\bigl(\tilde{f}_{0,1}(qz)+z\bigr)
	+\tilde{f}_{1,0}(qz)\bigl(\tilde{f}_{0,1}(pz)+z\bigr)\\
%	+z\tilde{f}_{1,0}(pz)+z\tilde{f}_{1,0}(qz)\\
	&\qquad+pz\tilde{f}'_{1,0}(pz)+qz\tilde{f}'_{1,0}(qz)
	+\tilde{f}_{0,1}(pz)
	+\tilde{f}_{0,1}(qz)+z(1-e^{-z}).
\end{align*}
To compute the covariance, it is beneficial to introduce now the
\emph{Poissonized covariance} (see \eqref{VS-VK} or \cite{FuHwZa}
for similar details)
\[
	\tilde{C}(z)=\tilde{f}_{1,1}(z)
	-\tilde{f}_{1,0}(z)\tilde{f}_{0,1}(z)
	-z\tilde{f}'_{1,0}(z)\tilde{f}'_{0,1}(z),
\]
which satisfies
\begin{align}\label{Cz}
	\tilde{C}(z)
	=\tilde{C}(pz)+\tilde{C}(qz)+\tilde{h}_1(z)+\tilde{h}_2(z),
\end{align}
where
\[
	\tilde{h}_1(z)
	=pqz\bigl(\tilde{f}'_{1,0}(pz)-\tilde{f}'_{1,0}(qz)\bigr)
	\bigl(\tilde{f}'_{0,1}(pz)-\tilde{f}'_{0,1}(qz)\bigr),
\]
and
\begin{align*}
	\tilde{h}_2(z)
	&=ze^{-z}\bigl(\tilde{f}_{1,0}(pz)+\tilde{f}_{1,0}(qz)
	+p(1-z)\tilde{f}'_{1,0}(pz)
	+q(1-z)\tilde{f}'_{1,0}(qz)\bigr)\\
	&+e^{-z}\bigl((1+z)\tilde{f}_{0,1}(pz)
	+(1+z)\tilde{f}_{0,1}(qz)
	-pz^2\tilde{f}'_{0,1}(pz)-qz^2\tilde{f}'_{0,1}(qz)
	\bigr)\\&+ze^{-z}\bigl(1-(1+z^2)e^{-z}\bigr).
\end{align*}
Note that $\tilde{h}_1$ is zero when $p=\frac12$. Furthermore, from
(\ref{exp-means}) (which can be differentiated since they hold in a
sector $\mathscr{S} =\{z\in\mathbb{C}\ :\ \Re(z)\geq\epsilon,
\vert{\rm Arg}(z)\vert\le\theta_0\}$ with $0<\theta_0<\pi/2$ in the
complex plane), we obtain that $\tilde{h}_1(z)=O(|z|)$ and
$\tilde{h}_2(z)$ is exponentially small for large $|z|$ in
$\Re(z)>0$. Also $\tilde{h}_1(z)+\tilde{h}_2(z)=O(|z|^2)$ as $z\to0$.
Thus the Mellin transform of $\tilde{h}_1(z)+\tilde{h}_2(z)$ exists
in the strip $\langle -2,0\rangle$, and we have then the inverse
Mellin integral representation
\[
	\tilde{C}(z) =
	\frac1{2\pi i}\int_{-\frac32-i\infty}
	^{-\frac32+i\infty}
	\frac{\mathscr{M}[\tilde{h}_1(z)
	+\tilde{h}_2(z);s]}{1-p^{-s}-q^{-s}}\,
	z^{-s}\mathrm{d} s,
\]
where $\mathscr{M}[\phi(z);s] := \int_0^\infty \phi(z)
z^{s-1}\mathrm{d} z$ denotes the Mellin transform of $\phi$.

We then show that $\mathscr{M}[\tilde{h}_1(z);s]$ can be analytically
continued to the vertical line $\Re(s)=-1$ and has no singularities
there. This is the most complicated part of the proof because
$\tilde{h}_1(z)$ contains the product of the two terms
$\tilde{f}'_{1,0}(pz) -\tilde{f}'_{1,0}(qz)$ and
$\tilde{f}'_{0,1}(pz) -\tilde{f}'_{0,1}(qz)$ and thus
$\mathscr{M}[\tilde{h}_1(z);s]$ becomes a Mellin convolution
integral. In \cite{FuHwZa}, a general procedure was given for the
simplification of such integrals (see \cite[p.\ 24 \emph{et
seq.}]{FuHwZa}). This simplification procedure and a direct
application of the theory of admissible functions of analytic
de-Poissonization now yield \begin{pro}\label{fc-cov} The covariance
of $S_n$ and $K_n$ is asymptotically linear:
\[
    \mathrm{Cov}(S_n,K_n)\sim n\mathscr{F}[g^{(2)}](n).
\]
Here
\begin{align}\label{g2k}
\begin{split}
	g^{(2)}_k
	&=\frac{\Gamma(\chi_k)}{h}\Bigl(1-\frac{\chi_k+2}
	{2^{\chi_k+1}}\Bigr)-\frac{1}{h^2}
	\sum_{j\in\mathbb{Z}\setminus\{0\}}
	\Gamma(\chi_{k-j}+1)(\chi_j-1)\Gamma(\chi_j)\\
	&-\frac{\Gamma(\chi_k+1)}{h^2}
	\Bigl(\gamma+1+\psi(\chi_k+1)
	-\frac{p\log^2p+q\log^2q}{2h}\Bigr)\\
	&+\frac{1}{h}\sum_{\ell\ge 2}
	\frac{(-1)^{\ell}(p^{\ell}+q^{\ell})}
	{\ell!(1-p^{\ell}-q^{\ell})}\,\Gamma(\chi_k+\ell-1)
	(2\ell^2-2\ell+1+\chi_k(2\ell-1)),
\end{split}
\end{align}
where $\gamma$ denotes Euler's constant, $\psi(z)$ is the digamma
function and $\chi_k$ is defined in \eqref{gk}.
\end{pro}

\begin{Rem}
If $\frac{\log p}{\log q}\not\in\mathbb{Q}$, then only $k=0$ is
relevant and the second term (the sum over $j$) on the right-hand
side of \eqref{g2k} has to be dropped. Also the first term here
$\frac{\Gamma(\chi_k)}{h} \bigl(1-\frac{\chi_k+2}
{2^{\chi_k+1}}\bigr)$ is taken to be its limit $\frac1h(\log
2+\frac12)$ as $\chi_k\to0$ when $k=0$.
\end{Rem}

The asymptotic estimate for the correlation coefficient in Theorem
\ref{main-thm-1} now follows from this and the results for the the
variances of $S_n$ and $K_n$ (see Table \ref{tb-SKN}), where
expressions for $g^{(1)}_k$ and $g^{(3)}_k$ can be found, e.g., in
\cite{FuHwZa}. For convenience, we give below the expressions in the
unbiased case. Note that both $\mathscr{F}[g^{(1)}](n)$ and
$\mathscr{F}[g^{(3)}](n)$ are strictly positive; see Schachinger
\cite{Sch} for details.

When $p=\frac12$, an alternative expression to \eqref{g2k} (avoiding
the convolution of two Fourier series) is
\begin{align*}
	g^{(2)}_k
	&=\frac{\Gamma(\chi_k)\Bigl(1-\frac{\chi_k^2+\chi_k+4}
	{2^{\chi_k+2}}\Bigr)}{\log 2}
	+\frac1{\log 2}\sum_{\ell\ge 1}\frac{(-1)^{\ell}
	\Gamma(\chi_k+\ell)
	\left(\ell(2\ell+1)(\chi_k+\ell)-(\ell+1)^2\right)}
	{(\ell+1)!(2^{\ell}-1)};
\end{align*}
see the discussion of the size of tries in \cite{FuHwZa}, where a
similar alternative expression was given for $g^{(1)}_k$, which reads
\begin{align*}
	g^{(1)}_k
	=-\frac{\Gamma(\chi_k-1)\chi_k(\chi_k+1)^2}{4\log 2}
	+\frac{2}{\log 2}\sum_{\ell\ge 1}\frac{(-1)^{\ell}
	\Gamma(\chi_k+\ell)
	\ell \bigl(\ell(\chi_k+\ell)-1\bigr)}
	{(\ell+1)!(2^{\ell}-1)}.
\end{align*}
Moreover, also in \cite{FuHwZa}, the following expression for
$g^{(3)}_k$ can be found
\begin{align*}
	g^{(3)}_k
	=\frac{\Gamma(\chi_k)\Bigl(1-\frac{\chi_k^2-\chi_k+4}
	{2^{\chi_k+2}}\Bigr)}{\log 2}
	+\frac{2}{\log 2}\sum_{\ell\ge 1}
	\frac{(-1)^{\ell}\Gamma(\chi_k+\ell)(\ell(\chi_k+\ell-1)-1)}
	{\ell!(2^{\ell}-1)}.
\end{align*}
Note that $\chi_k=\frac{2k\pi i}{\log 2}$ and $2^{\chi_k}=1$, and the
reason of retaining $2^{\chi_k+2}$ in the denominator is to give a
uniform expression for all $k$ (notably $k=0$). These provide an
explicit expression for the periodic function $F(n)$ in Theorem
\ref{main-thm-1}. Also, since all the periodic functions have very
small amplitude, the average value of the periodic function $F(z)$ can
be well-approximated by
\[
	\frac{g^{(2)}_0}{\sqrt{g^{(1)}_0g^{(3)}_0}}
	\approx 0.9272416035\cdots.
\]

\section{Limit Law}\label{ll}

In this section, we prove Theorem \ref{main-thm-2}, part (i); the
proof of part (ii) is similar and skipped here. The key tool of the
proof is the multivariate version of the contraction method; see
Neininger and R\"{u}schendorf \cite{NeRu2}. More precisely, we will
use Theorem 3.1 in \cite{NeRu2}.

We first recall the expression for the square-root of a
positive-definite $2\times 2$ matrix $M=\begin{pmatrix} a & b \\ b &
c\end{pmatrix}$. It is well-known that such a matrix has exactly one
positive-definite square root which is given by
\[
	M^{\frac12}=\frac{1}{\sqrt{a+c+2\sqrt{ac-b^2}}}
	\begin{pmatrix}
		a+\sqrt{ac-b^2} & b \\
		b & c+\sqrt{ac-b^2}
	\end{pmatrix},
\]
with the inverse
\begin{align}\label{A-half}
	M^{-\frac12}=\frac1
	{\sqrt{(ac-b^2)\bigl(a+c+2\sqrt{ac-b^2}\bigr)}}
	\begin{pmatrix}
		c+\sqrt{ac-b^2} & -b \\
		-b & a+\sqrt{ac-b^2}
	\end{pmatrix}.
\end{align}
Now we sketch the proof of Theorem \ref{main-thm-2}, Part (i).

\paragraph{Proof of Theorem \ref{main-thm-2}, Part (i).}
First note that
\[
	\begin{pmatrix}
		S_n \\
		K_n
	\end{pmatrix}
	\stackrel{d}{=}
	\begin{pmatrix}
	    1 & 0 \\
		0 & 1
	\end{pmatrix}
	\begin{pmatrix} S_{B_n}\\ K_{B_n}\end{pmatrix}
	+\begin{pmatrix} 1 & 0 \\ 0 & 1\end{pmatrix}
	\begin{pmatrix} S_{n-B_n}^{*}\\ K_{n-B_n}^{*}\end{pmatrix}
	+\begin{pmatrix} 1 \\ n\end{pmatrix},
\]
where the notation is as in Section~\ref{intro}. The contraction
method was specially developed for obtaining limiting distribution
results for such recurrences; see \cite{NeRu2}.

We need some notation. First, define
\begin{align}\label{Sigma-n}
	\widehat{\Sigma}_n:=\begin{pmatrix}
	    \mathbb{V}(S_n) & {\rm Cov}(S_n,K_n) \\
		{\rm Cov}(S_n,K_n) & \mathbb{V}(K_n)
	\end{pmatrix}.
\end{align}
This matrix is clearly positive-definite for all $n$ sufficiently
large. Next define
\[
	M_n^{(1)}:=\widehat{\Sigma}_n^{-\frac12}\widehat{\Sigma}_{B_n}^{\frac12},\qquad
	M_n^{(2)}:=\widehat{\Sigma}_n^{-\frac12}\widehat{\Sigma}_{n-B_n}^{\frac12}
\]
and
\[
	\begin{pmatrix} b_n^{(1)} \\ b_n^{(2)}\end{pmatrix}
	=\widehat{\Sigma}_n^{-\frac12}\begin{pmatrix}
	1-\mu(n)+\mu(B_n)+\mu(n-B_n)\\
	n-\nu(n)+\nu(B_n)+\nu(n-B_n)
	\end{pmatrix},
\]
where $\mu(n)={\mathbb E}(S_n)$ and $\nu(n)={\mathbb E}(K_n)$.

Now to apply the contraction method in \cite{NeRu2}, it suffices
to show that the following conditions hold
\begin{align}
	&b_n^{(i)}\stackrel{L_3}{\longrightarrow} 0,
	\qquad M_n^{(i)}\stackrel{L_3}{\longrightarrow} M_i,
	\label{cond-1}\\
	&\mathbb{E}\bigl(\|M_1\|^3_{\text{op}}
	+\|M_2\|^3_{\text{op}}\bigr)<1,
	\qquad\mathbb{E}\bigl(\|M_n^{(i)}\|^3_{\text{op}}
	\chi_{\{B_n^{(i)}\le j\}\cup\{B_n^{(i)}=n\}}\bigr)
	\longrightarrow 0\label{cond-2}
\end{align}
for $i=1,2$ and $j\in\mathbb{ N}$, where $\stackrel{L_3}
{\longrightarrow}$ denotes convergence in the $L_3$-norm,
$\|\cdot\|_{\text{op}}$ is the operator norm, $\chi_S$ denotes the
characteristic function of set $S$, $B_n^{(1)}=B_n, B_n^{(2)}=n-B_n$
and
\[
	M_1=\begin{pmatrix}
	    \sqrt{p} & 0 \\ 0 & \sqrt{p}
	\end{pmatrix},\qquad
	M_2=\begin{pmatrix}
	    \sqrt{q} & 0 \\ 0 & \sqrt{q}
	\end{pmatrix}.
\]
Then the contraction method in \cite{NeRu2} guarantees that
$(S_n,K_n)$ (centralized and normalized) converges in distribution to
the unique fixed-point with mean $0$, covariance matrix the unity
matrix and finite $L_3$-norm of
\[
	\begin{pmatrix}X_1\\ X_2\end{pmatrix}
	\stackrel{d}{=}
	\begin{pmatrix}\sqrt{p} & 0 \\ 0 & \sqrt{p}\end{pmatrix}
	\begin{pmatrix}X_1\\ X_2\end{pmatrix}
	+\begin{pmatrix}\sqrt{q} & 0 \\ 0 & \sqrt{q}\end{pmatrix}
	\begin{pmatrix}X_1^{*}\\ X_2^{*}\end{pmatrix},
\]
where $(X_1^{*},X_2^{*})$ is an independent copy of $(X_1,X_2)$.
Obviously, the bivariate normal distribution is the solution.
All this is summarized as follows.

\begin{pro}\label{asym-clt} The following convergence in distribution
holds:
\[
	\widehat{\Sigma}_n^{-\frac12}
	\begin{pmatrix}
		S_n-\mathbb{E}(S_n)\\
		K_n-\mathbb{E}(K_n)
	\end{pmatrix}
		\stackrel{d}{\longrightarrow} \mathcal{N}_2(0, I_2).
\]
\end{pro}
\begin{proof}
We only check (\ref{cond-1}) because the second condition of (\ref{cond-2}) follows along similar lines and the first condition of (\ref{cond-2}) follows from \eqref{cond-1} in view of
\[
	\|M_1\|_{\text{op}}=\sqrt{p}\qquad\text{and}
	\qquad \|M_2\|_{\text{op}}=\sqrt{q}.
\]

We start with proving the claimed property for $b_n^{(i)}$ for which
we use the notations
\[
	\Omega_1(n)=\mathbb{V}(S_n),
	\quad\Omega_2(n)={\rm Cov}(S_n,K_n),
	\quad\Omega_3(n)=\mathbb{V}(K_n)
\]
and
\[
    D(n)=\Omega_1(n)\Omega_3(n)-\Omega_2(n)^2.
\]
Also define
\[
    R(n) = \Omega_1(n)+\Omega_3(n)+2\sqrt{D(n)}.
\]

Then, by \eqref{A-half}, we see that
\begin{align*}
	b_n^{(1)}&=(1-\mu(n)+\mu(B_n)
	+\mu(n-B_n))\frac{\Omega_3(n)
	+\sqrt{D(n)}}{\sqrt{D(n)R(n)}}\\
	&\quad-(n-\nu(n)+\nu(B_n)
	+\nu(n-B_n))\frac{\Omega_2(n)}{\sqrt{D(n)R(n)}}
\end{align*}
and a similar expression for $b_n^{(2)}$ holds. From the normality of
both $S_n$ and $K_n$ (proved for $S_n$ via the contraction method in
\cite{FuLe2} and a similar method of proof also applies to $K_n$), we
have
\[
	\frac{1-\mu(n)+\mu(B_n)+\mu(n-B_n)}{\sqrt{n}}
	\stackrel{L_3}{\longrightarrow} 0 \qquad\text{and}\qquad
	\frac{n-\nu(n)+\nu(B_n)+\nu(n-B_n)}{\sqrt{n\log n}}
	\stackrel{L_3}{\longrightarrow} 0.
\]
Moreover, we have
\[
	\sqrt{n}\,\frac{\Omega_3(n)+\sqrt{D(n)}}{\sqrt{D(n)R(n)}}
	\sim\frac{1}
	{\sqrt{\mathscr{F}[g^{(1)}](n)}},
\]
and
\[
	\sqrt{n\log n}\,\frac{\Omega_2(n)}
	{\sqrt{D(n)R(n)}}
	\sim\frac{\mathscr{F}[g^{(2)}](n)}
	{\lambda\sqrt{\log n\mathscr{F}[g^{(1)}](n)}},
\]
where $g^{(1)}, g^{(2)}$ and $\lambda$ are as above. Thus, both
sequences are bounded and, consequently, we obtain the claimed result
with $L_3$-convergence above. Similarly, one proves the claimed
result for $b_n^{(2)}$.

Next, we consider $M_n^{(i)}$. Here, we only show the claim for the
$(1,1)$ entry of $M_n^{(1)}$ (denoted by $M_n^{(1)}(1,1)$) all other
cases being treated similarly. First, observe that by definition and
matrix square-root, we have
\begin{align*}
	M_n^{(1)}(1,1)
	=\frac{\sqrt{R(n)}}{\sqrt{R(B_n)}}
	\cdot\frac{(\Omega_3(n)
	+\sqrt{D(n)})(\Omega_1(B_n)+\sqrt{D(B_n)})
	-\Omega_2(n)\Omega_2(B_n)}{\sqrt{D(n)R(n)}}.
\end{align*}
Now, from the strong law of large numbers for the binomial
distribution
\[
    \frac{B_n}{n}\stackrel{\text{a.s.}}{\longrightarrow} p
\]
and from Taylor series expansion (note that all periodic functions are
infinitely differentiable), we have
\[
	\frac{\sqrt{R(n)}}{\sqrt{R(B_n)}}
	\stackrel{\text{a.s.}}{\longrightarrow}\frac{1}{\sqrt{p}},
\]
and
\[
	\frac{(\Omega_3(n)
	+\sqrt{D(n)})(\Omega_1(B_n)+\sqrt{D(B_n)})
	-\Omega_2(n)\Omega_2(B_n)}{\sqrt{D(n)R(n)}}
	\stackrel{\text{a.s.}}{\longrightarrow}p.
\]
Thus, $M_n^{(1)}(1,1)\stackrel{\text{a.s.}}
{\longrightarrow} \sqrt{p}$ from which the claim follows by the
dominated convergence theorem.
\end{proof}

Next, set
\[
    \widetilde{\Sigma}_n
	:=\begin{pmatrix}
	    n\mathscr{F}[g^{(1)}](n) & 0 \\
		0 & \lambda n\log n
	\end{pmatrix}.
\]
Then, we have the following simple lemma.
\begin{lmm}
We have, as $n\rightarrow\infty$,
\[
    \widehat{\Sigma}_n^{-\frac12}\widetilde{\Sigma}_n^{\frac12}
	\rightarrow I_2.
\]
\end{lmm}
\begin{proof} This follows by a straightforward computation using the
expressions of the matrix square-root and its inverse from above.
\end{proof}

%We will exemplify the computation for one entry of $\widehat{\Sigma}_n^{-\frac12}
%\widetilde{\Sigma}_n^{\frac12}$, where we use the notations from the
%proof of the previous proposition.

%For instance, for the $(1,2)$-entry of $\widehat{\Sigma}_n^{-\frac12}
%\widetilde{\Sigma}_n^{\frac12}$, we have
%\[
%	-\frac{\Omega_2(n)\sqrt{\Omega_3(n)}}
%	{\sqrt{D(n)R(n)}}
%	\sim-\frac{\mathscr{F}[g^{(2)}](n)}
%	{\sqrt{\lambda\log n\mathscr{F}[g^{(1)}](n)}}
%\]
%which tends to $0$ as claimed.

From this lemma and Proposition \ref{asym-clt} our claimed result now
follows.

\section*{Acknowledgments}

We thank the referees for helpful comments; in particular, the
detailed comments by one referee on the proof of Theorem
\ref{main-thm-2} led to considerable improvements in Section \ref{ll}.

\end{document}